\documentclass[12pt,reqno]{amsart}
\usepackage{amsmath,amssymb,amstext,amsfonts,epsfig,graphicx}
\usepackage[margin=1in]{geometry}
\usepackage{float}
\usepackage{graphicx}
\usepackage{subcaption}
\usepackage{morefloats}
\usepackage{epsfig}
\usepackage{morefloats}
\usepackage{mathtools}
\usepackage{amsmath}
\usepackage{amsfonts}
\usepackage{amssymb}
\usepackage{amsthm}
\usepackage{MnSymbol}
\usepackage[usenames,dvipsnames]{color}

\newtheorem{theorem}{Theorem}
\newtheorem{lemma}[theorem]{Lemma}

\newtheorem{conjecture}[theorem]{Conjecture}

\newtheorem{algorithm}[theorem]{Algorithm}
\newtheorem{corollary}[theorem]{Corollary}

\title{Burning a Graph is Hard}
\author{Anthony Bonato}
\address{Department of Mathematics\\
Ryerson University\\
Toronto, ON\\
Canada, M5B 2K3} \email{abonato@ryerson.ca}
\author{Jeannette Janssen}
\address{Department of Mathematics and Statistics\\
Dalhousie University\\
Halifax, NS\\
Canada, B3H 3J5}
\email{jeannette.janssen@dal.ca}
\author{Elham Roshanbin}
\address{Department of Mathematics and Statistics\\
Dalhousie University\\
Halifax, NS\\
Canada, B3H 3J5}
\email{e.roshanbin@Dal.Ca}
\thanks{Supported by grants from NSERC}

\begin{document}
\begin{abstract} 
Graph burning is a model for the spread of social contagion. The burning number is a graph parameter associated with graph burning that measures the speed of the spread of contagion in a graph; the lower the burning number, the faster the contagion spreads. We prove that the corresponding graph decision problem is \textbf{NP}-complete when restricted to acyclic graphs with maximum degree three, spider graphs and path-forests. We provide polynomial time algorithms for finding the burning number of spider graphs and path-forests if the number of arms and components, respectively, are fixed.
\end{abstract}
\maketitle

\section{Introduction}

Suppose you were attempting to spread gossip, a meme, or some other social contagion in an online social network such as Facebook or Twitter. Our assumptions, similar to those in the recent study on the spread of emotional contagion in Facebook~\cite{kramer}, are that in-person interaction and nonverbal cues are not necessary for the spread of the contagion. Hence, agents in the network spread the contagion to their friends or followers, and the contagion propagates over time. If the goal was to minimize the time it took for the contagion to reach the entire network, then which users would you target with the contagion, and in which order?  Related questions emerge in study of the spread of social influence, which is an active topic in social network analysis; see, for example, \cite{dr,kkt,ktt1,mr,rd}. \emph{Graph burning} is a simplified deterministic model for the spread of social contagion in a social network that considers an answer to these questions, and was introduced in \cite{bonato2}. 

Graph burning is a newly discovered deterministic graph process
in which we attempt to burn all the nodes as quickly as possible, and is inspired by contact processes on graphs such as graph bootstrap percolation \cite{bbm}, and graph searching paradigms such as Firefighter \cite{bonato, fmac}. Throughout, we work with simple, undirected, and finite graphs. There are discrete time-steps or rounds. At time $t=0$ all the nodes are unburned. Then at each time $t\geq 1$, we burn one new unburned node if such a node is available. Once a node is burned in round $t$, each of its unburned neighbours becomes burned in round $t+1$. If a node is burned, then it remains in that state until the end of the process. The process ends when all nodes are burned. 

We denote the node that we burn in the $i$-th step by $x_i$, and we call it a \emph{source of fire}.  
If we burn a graph $G$ in $k$ steps, then the sequence $(x_1, x_2,\ldots, x_k)$ is called a \emph{burning sequence} for $G$. The \emph{burning number} of $G$, written by $b(G)$, is the length of a shortest burning sequence for $G$; such a burning sequence is referred to as \emph{optimum}. In other words, the burning number of $G$ is the minimum number of steps needed for the burning process to end. 

The following graph decision problem is our main focus in this paper.

\medskip

\noindent{\bf Problem: Graph Burning}

\noindent{\bf Instance:} A simple graph $G$ of order $n$ and an integer $k\geq 2$.

\noindent{\bf Question:} Is $b(G) \leq k$? In other words, does $G$ contain a burning sequence $(x_1, x_2, \ldots, x_k)$?
\medskip

As we show in Lemma 14 from \cite{bonato2}, the burning number of a graph is tightly bounded in terms of its distance domination numbers. We also proved in Corollary 5 from \cite{bonato2} that the burning number of a connected graph $G$ is the minimum burning number over the set of spanning subtrees of $G$. On the other hand, the burning problem has some similarities to some other known graph processes such as Firefighter.
It is known that the distance domination problems are polynomially solvable for trees (see \cite{dbook}), however, the Firefighter problem is \textbf{NP}-complete even for trees of maximum degree three (see \cite{fking}). These are our motivation to investigate the complexity of the burning problem for graphs; in particular for trees.

This paper is organized as follows. We first provide a short review on the basic results that we have found on the burning number in \cite{bonato2, thesis}. These results are needed in this paper. In Section $1$, we prove that the Graph Burning problem is \textbf{NP}-complete when restricted to trees of maximum degree three. As a corollary, this shows the \textbf{NP}-completeness of the burning problem for chordal graphs, bipartite graphs, planar graphs, disconnected graphs, and binary trees. As a special example of trees with maximum degree three, we find the burning number of perfect binary trees. In Section $2$, we show that the burning problem remains \textbf{NP}-complete even for trees with a structure as simple as spider graphs, and also for disconnected graphs such as path-forests. 
In Section $3$, we provide polynomial algorithms for finding the burning number of path-forests and spider graphs, when the number of arms and components is fixed.

\subsection{Preliminaries}
In this subsection, we give some background and notation from graph theory. For further background, see \cite{west}. Then we review some facts about the burning problem from \cite{bonato2, thesis}  that will be useful in the present study.

If $v$ is a node of a graph $G$, then the \emph{eccentricity} of $v$ is defined as $\max\{d(v, u): u\in V(G) \}$. The \emph{center} of $G$ consists of the nodes in $G$ with minimum eccentricity. Every node in the center of $G$ is called a \emph{central} node of $G$. The \emph{radius} of $G$ is the minimum eccentricity over the set of all nodes in $G$. The \emph{diameter} of $G$ is the maximum eccentricity over the set of all nodes in $G$. Given a positive integer $k$, the \emph{$k$-th closed neighborhood} of $v$ is defined to be the set $\{u\in V(G) : d(u, v)\leq k\}$ and is denoted by $N_k[v]$; we denote $N_1[v]$ simply by $N[v]$. 
We call a graph a \emph{path-forest} if it is the disjoint union of a collection of paths. 
For $s\ge 3$, let $K_{1, s}$ denote a \emph{star}; that is, a complete bipartite graph with parts of order $1$ and $s$. 

We call a
tree that has only one node $c$ of degree at least three a \emph{spider graph}, and the node $c$ is called the \emph{spider head}. In a spider graph every leaf is connected to the spider head by a path which is called an
\emph{arm}. If all the arms of a spider graph with maximum degree $s$ are of the same length $r$, then we denote such a spider graph by $SP(s,r)$.
We denote the \emph{disjoint union} of two graphs $G$ and $H$ by $G\cupdot H$.

The following facts about graph burning can be found in \cite{bonato2,thesis}.
In a graph $G$ the sequence $(x_1, x_2, \ldots, x_k)$ forms a burning
sequence if and only if, for each pair $i$ and $j$ with $1\leq i < j\leq k$, we have that $d(x_i, x_j) \geq j-i$, and the following set equation holds:
\begin{equation}
N_{k-1}[x_1] \cup N_{k-2}[x_2]\cup \ldots \cup N_0[x_k] = V(G).\label{eqq}
\end{equation}

\begin{corollary}\label{cov3}
If $(x_1, x_2, \ldots, x_k)$ is a sequence of nodes in a graph $G$, such that $N_{k-1}[x_1] \cup N_{k-2}[x_2]\cup \ldots \cup N_0[x_k] = V(G)$, then $b(G) \leq k$.
\end{corollary}

We proved in \cite{bonato2} (Lemma 11) that the burning number of a graph $G$ of radius $r$ is at most $r+1$.
Namely, by burning a central node of $G$ at the first step, every other node will be burned after at most $r$ more steps.

\begin{theorem}\cite{bonato2}\label{path}
For a path $P_n$ we have that $b(P_n) = \lceil n^{1/2}\rceil$.
More precisely, if $n = k^2$ for some integer $k$, then burning $P_n$ in $k$ steps is equivalent to decomposing $P_n$ into $k$ subpaths of orders $1,3, \ldots, 2k-1$.
If $\lceil n^{1/2}\rceil = k$, and $n$ is not a square number, then every optimum burning sequence for $P_n$ corresponds to a decomposition of $P_n$ into $k$ smaller subpaths $Q_1, Q_2, \ldots, Q_k$, in which the order of each $Q_i$ is a number between one and $2i-1$.
\end{theorem}

\begin{theorem}\cite{thesis}\label{path-forest}
If $G$ is a path-forest of order $n$ with $t\geq 1$ components, then 
$$b(G) \leq \lceil n^{1/2}\rceil + t-1.$$
\end{theorem}

A subgraph $H$ of graph $G$ is called an \emph{isometric} subgraph of $G$ if the distance between any pair of vertices $u$ and $v$ in $H$ equals the distance between $u$ and $v$ in $G$.
For example, any subtree of a tree $T$ is an isometric subgraph of $T$. The following corollary is a generalization of Theorem $7$ in \cite{bonato2} for disconnected graphs.

\begin{corollary}\cite{thesis}\label{subforest}
If $G$ is a graph with components $G_1, G_2, \ldots, G_t$, and $H$ is an isometric subforest of $G$, then we have that $b(H)\leq b(T).$
\end{corollary}

Note that the only graph with burning number one is $K_1$. Moreover, the following theorem characterizes the graphs with burning number $2$.
\begin{theorem}\cite{thesis}\label{b222}
If $G$ is a graph of order $n$, then $b(G)=2$ if and only if $n\geq 2$, and $G$ has maximum degree $n-1$ or $n-2.$
\end{theorem}

Since finding the maximum degree of a graph is solvable in polynomial time, 
by Theorem~\ref{b222}, we can recognize graphs with burning number $2$ in polynomial time. Thus, in the rest of the paper we restrict our attention to the case $k\geq 3$.

\section{Burning Trees and Forests with Maximum Degree Three}

We now consider Graph Burning in acyclic graphs with maximum degree three. In particular, we show in Theorem \ref{3tree} that the burning problem is NP-complete for trees of maximum degree three.  Then as a special example of trees with maximum degree three, we find the burning number of perfect binary trees.
We show the \textbf{NP}-completeness of the burning problem by a reduction from a variant of the $3$-Partition problem \cite{gary}. Here is the statement of this problem.

\medskip
\noindent{\bf Problem: Distinct $3$-Partition}

\noindent{\bf Instance:} A finite set $X =\{a_1, a_2, \ldots, a_{3n}\}$ of positive distinct integers, and a positive integer $B$ where $\sum_{i=1}^{3n} a_i = nB$, and $B/4 < a_i < B/2$, for $1\leq i\leq 3n$.

\noindent{\bf Question:} Is there any partition of $X$ into $n$ triples such that the elements in each triple add up to $B$?
\medskip

In \cite{hulet}, it is shown that the Distinct $3$-Partition problem is \textbf{NP}-complete \emph{in the strong sense} (see \cite{gary}); that is, this problem is \textbf{NP}-complete, even when restricted to the cases where $B$ is bounded above by a polynomial in $n$. In the rest of the paper, by $O_m$, we mean the set of the $m$ first positive odd integers; that is, $O_m = \{1,3, \ldots, 2m-1\}$.

\begin{theorem}\label{3tree}
The burning problem is \textbf{NP}-complete for trees of maximum degree three.
\end{theorem}

\begin{proof}
Given a graph $G$ of order $n$ and a sequence $(x_1, x_2, \ldots, x_k)$ of the nodes in $G$, we can easily find $N_{k-i}[x_i]$ in polynomial time, for $1\leq i\leq k$. Thus, we can check in polynomial time if $V(G) = \bigcup_{i=1}^k N_{k-i}[x_i]$. Hence, the burning problem is in \textbf{NP}. 

Now, we show the \textbf{NP}-completeness of the burning problem for trees of maximum degree three by a reduction from the Distinct $3$-Partition problem.

Suppose that we have an instance of the Distinct $3$-Partition problem; that is, we are given a non-empty finite set $X =\{a_1, a_2, \ldots, a_{3n}\}$ of distinct positive integers, and a positive integer $B$ such that $\sum_{i=1}^{3n} a_i = nB$, and $B/4 < a_i < B/2$, for $1\leq i\leq 3n$. Since the Distinct $3$-Partition problem is \textbf{NP}-complete in the strong sense, without loss of generality we can assume that $B$ is bounded above by a polynomial in $n$.
Assume that the maximum of the set $X$ is $m$ which is by assumption bounded above by $B/2$. We now construct a tree of maximum degree $3$ as follows.

Let $Y = \{2a_i -1 : a_i \in X \}$. Hence, $Y  \subseteq O_m$, and $2nB +3n = \sum_{i=1}^{3n} 2a_i -1$ is the sum of the numbers in $Y$. Let $Z = O_m \setminus Y$.
Note that $1\leq |Y| \leq m$, and consequently, $|Z| \leq m - 1$. 
Let $|Z| = k$, for some $k\leq m- 1$. For $1\leq i\leq k$, let $Q'_i$ be a path of order $l_i$, where $l_i$ is the $i$-th largest number in $Z$. 
For $1\leq i\leq m+1$, we define $T_i$ to be a spider $SP(3, 2m+1 -i)$ with centre $r_i$.
We also take $Q_i$ to be a path of order $2B+3$, for $1\leq i\leq n$. 
Then we combine the graphs that we created above from left to right in the following order: 
$$Q_1, T_1, Q_2, T_2,  \ldots, Q_n, T_n, Q'_1, T_{n+1}, Q'_2, T_{n+2}, \ldots, Q'_k, T_{n+k}, T_{n+k+1},  \ldots , T_{m+1}$$ 
such that each graph in this order is joined by an edge from one of its leaves to a leaf of the next graph in the presented order. The resulting graph is named $T(X)$; note that it is a tree of maximum degree three.

For example, let $X = \{10, 11, 12, 14, 15, 16\}$, and $B = 39$. Then $n=2$, and $m = \max\{a_i: a_i \in X\} = 16 $. Therefore, $Y = \{19, 21, 23, 27, 29, 31\}$, and $Z = O_{16} \setminus Y = \{1,3,5,7,9, 11, 13, 15, 17, 25\}$. Thus, $k = |Z| = |O_{16} \setminus Y| = 10$.
The graph $T(X)$ is depicted in Figure \ref{max3}. Due to the space limits, we do not draw the nodes in the paths $Q_i$ and $Q'_j$ in this figure.

\begin{figure}[H]
\begin{center}
\includegraphics[scale=0.7]{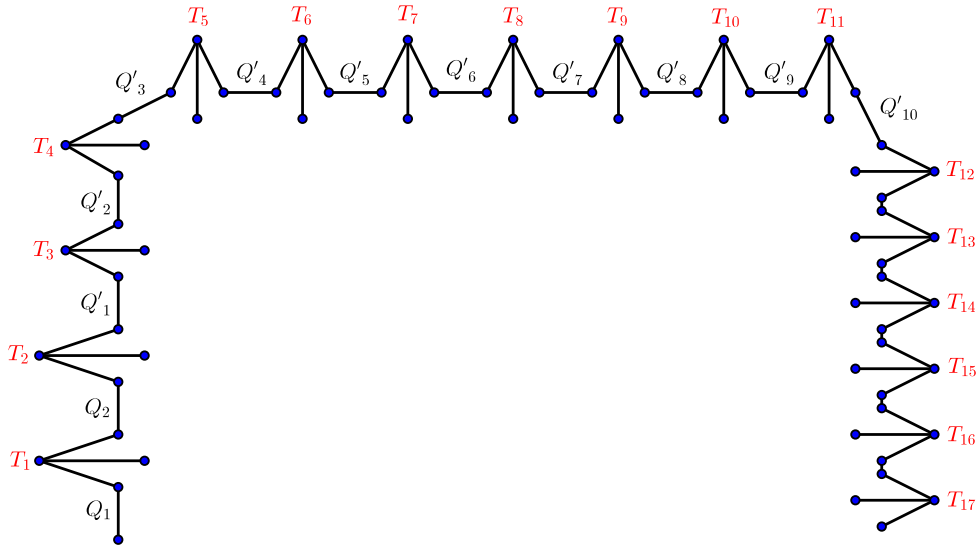}
\end{center}
\caption{A sketch of the tree $T(X)$.}\label{max3}
\end{figure}

For $1\leq i\leq m+1$, let $v_i$ be a leaf of $T(X)$ that is also a leaf of $T_i$, as a subgraph of $T(X)$. Note that for $1\leq i\leq m+1$ the two arms of $T_i$ that do not contain $v_i$, together with its centre $r_i$, form a path. We call this path $T'_i$. The order of $T'_i$ is $2(2m+1-i) +1 \in O_{2m+1}$. Hence, the subgraph of $T(X)$ induced by $$\left(\bigcup_{i=1}^{m+1}T'_i\right)  \bigcup  \left( \bigcup_{i=1}^{n} Q_i\right) \bigcup \left(\bigcup_{i=1}^k Q'_i\right) $$ forms a path of order $$\sum_{i=1}^{2m+1} (2i -1) = (2m+1)^2,$$
that we denote it by $P$. Therefore, 
$$ T(X) - P  = \bigcup_{i=1}^{m+1} (T_i - P), $$
which is a disjoint union of paths of orders $\{2m+1 - i\}_{i=1}^{m+1}$. Note that $T_i - P$ is the arm of $T_i$ that contains $v_i$. Thus, we have that 
$$|V(T(X))| = (2m+1)^2  + \sum_{i=1}^{m+1} (2m+1-i) = (2m+1)^2 + \frac{3(m^2 + m)}{2}.$$
Since $m$ is bounded by $B$ and by assumption, $B$ is bounded above by a polynomial in terms of $n$, then $T(X)$ is obtained in polynomial time in terms of the length of the input.

We can easily see that, there is a partition of $X$ into triples such that the elements in each triple add up to $B$ if and only if we can decompose the paths $Q_1, Q_2, \ldots, Q_n$ into subpaths of orders $2a_i -1 \in Y$. 
First, assume that there is a partition of $X$ into triples such that the elements in each triple add up to $B$. Equivalently, we have a partition for the paths $Q_1, Q_2, \ldots, Q_n$ in terms of subpaths $\{P_l: l\in Y\}$. Since $O_m = Y \cup Z$, we conclude that there is a partition for the subgraph $(\bigcup_{i=1}^n Q_i) \bigcup (\bigcup_{i=1}^{k} Q'_i)$ in terms of the subpaths $\{P_l : l\in O_m\}$.
Now, for $m+2 \leq i \leq 2m+1$, let $x_i$ be the centre of a path $P_l$ in such a partition, where $l = 2(2m+2-i)-1 \in O_m = Y \cup Z$. 
For $1\leq i\leq m+1$, let $x_i = r_i$ (the centre of $T_i$). 
Thus, we have that $$V(T(X)) = \bigcup_{i=1}^{2m+1} N_{2m+1-i}[x_i].$$
Consequently, by equation (\ref{eqq}), we conclude that $(x_1, x_2, \ldots ,x_{2m+1})$ forms a burning sequence of length $2m+1$ for $T(X)$. Therefore, $b(T(X)) \leq 2m+1$.

Conversely, suppose that $b(T(X)) \leq 2m+1$. 
Note that the path $P$ of order $(2m+1)^2$ is a subtree of $T(X)$. Therefore, by Theorem \ref{path} and Corollary \ref{subforest}, we have that $$b(T(X)) \geq b(P) = 2m+1.$$
Thus, we conclude that $b(T(X)) = 2m+1$. Assume that $(x_1, x_2, \ldots, x_{2m+1})$ is an optimum burning sequence for $T(X)$. 

We first claim that each $x_i$ must be in $P$. 
Since $T(X) = P \bigcup ( \bigcup_{j=1}^{m+1} (T_j - P) )$, every $x_i$ is either in $P$ or in $T_j \setminus P$, for some $1\leq j \leq m+1$. 
On the other hand, every node in $P$ must receive the fire from one of the $x_i$'s. Note that the only connection of $P$ to $T(X) \setminus P$ is through the nodes $r_i$'s. Hence, for $1\leq i\leq 2m+1$, $N_{2m+1-i}[x_i] \cap P$ must be a path of order at most $2(2m+1-i)+1$. If for some $1\leq i\leq 2m+1$, a node $x_i$  is out of $P$, then $N_{2m+1-i}[x_i] \cap P$ is a path of order less than $2(2m+1-i)+1$. Therefore, the total sum of the orders of the subpaths  $\{N_{2m+1-i}[x_i] \cap P\}_{i=1}^{2m+1}$ will be less than $(2m+1)^2  = |V(P)|$, which is a contradiction. Thus, every $x_i$ must be selected from $P$.

Now, we claim that for $1\leq i\leq m+1$, we must have $x_i = r_i$. We prove this by strong induction on $i$. Note that each $v_i$, $1\leq i\leq m+1$, receives the fire from a fire source $x_j \in P$ (by the above argument) where $1\leq j\leq 2m+1$. Therefore, $d(x_j, v_i) \leq 2m+1 -j$, for some $1\leq j\leq 2m+1$. 
For $i=1$, since the only node in $P$ that is within distance $2m+1 -i = 2m$ from $v_1$ is $r_1$, then we must have $x_1 = r_1$. Suppose that for $1\leq i\leq m$ and for every $1\leq j\leq i$, $x_j = r_j$. Since the only node in $P$ within distance $2m+1 - (i+1)$ from $v_{i+1}$ is the node $r_{i+1}$, and by induction hypothesis, we conclude that $x_{i+1} = r_{i+1}$. 
Therefore, the claim is proved by induction.

Note that $N_{2m+1 -i}[r_i] = V(T_i)$. Therefore,  the fire started at $x_i = r_i$ will burn all the nodes in $T_i$. 
The above argument implies that the nodes in $T(X) \setminus \bigcup_{i=1}^{m+1} T_i$ must be burned by receiving the fire started at $x_{m+2}, x_{m+3}, \ldots, x_{2m+1}$ (the last $m$ sources of fire). Since $T(X) \setminus \bigcup_{i=1}^{m+1} T_i$ is a disjoint union of paths, then we derive that for $m+2\leq i\leq 2m+1$, $N_{2m+1-i}[x_i] \bigcap (T(X) \setminus \bigcup_{i=1}^{m+1} T_i )$ is a path of order at most $2m+1- i$ ($\leq 2m-1$). 
On the other hand, the path-forest $T(X) \setminus \bigcup_{i=1}^{m+1} T_i$ is of order 
$$\sum_{i=1}^m (2m+1-i) = m^2.$$
Thus, we conclude that for $m+2\leq i\leq 2m+1$, $N_{2m+1-i}[x_i] \bigcap (T(X) \setminus \bigcup_{i=1}^{m+1} T_i )$ is a path of order equal to $2m+1-i$; since otherwise, we can not burn all the nodes in $T(X) \setminus \bigcup_{i=1}^{m+1} T_i$ in $m$ steps, which is a contradiction.
Therefore, there must be a partition of $T(X) \setminus \bigcup_{i=1}^{m+1} T_i$ (induced by the burning sequence $(x_{m+2}, x_{m+3}, \ldots, x_{2m+1})$) for $T(X) \setminus \bigcup_{i=1}^{m+1} T_i$) into subpaths $\{P_l: l \in O_m\}$.

Now, considering the partition described in the previous paragraph, we claim that there is a partition of $T(X) \setminus \bigcup_{i=1}^{m+1} T_i$ into subpaths of orders in $O_m$ in which the paths $Q_1, Q_2, \ldots, Q_n$ are decomposed into paths of orders in $Y$, and each path $Q'_i$ is covered by itself. 
Note that by definition, for $1\leq i\leq k$, each path $Q'_i$ is a component of $T(X) \setminus \bigcup_{i=1}^{m+1} T_i$. Hence, it suffices to prove that there is a partition of $T(X) \setminus \bigcup_{i=1}^{m+1} T_i$ into subpaths of orders in $O_m$ such that each $Q'_i$ is covered by itself. Assume that in a partition of $T(X) \setminus \bigcup_{i=1}^{m+1} T_i$ into subpaths of orders in $O_m$, there is a path $Q'_i$ of order $l\in O_m \setminus Y = Z $ that is partitioned by a union of paths of orders in $O_m$ rather than by $P_l$ itself. We know that $P_l$ must have covered some part of a path $Q'_j$ with $j\neq i$, or must be used in partitioning $Q_1, Q_2, \ldots, Q_n$. Hence, we can easily modify the partition by switching the place of $P_l$ and those paths that have covered $P_l$ (as they have equal lengths). Therefore, we have decreased the number of such displaced paths in our partition for $T(X) \setminus \bigcup_{i=1}^{m+1} T_i$. Since the number of $Q'_i$'s, where $1\leq i \leq k$,  is finite, we will end up after finite number of switching in a partition for $T(X) \setminus \bigcup_{i=1}^{m+1} T_i$ in which every $Q'_i$, $1\leq i\leq k$, is covered by itself.

Since each $Q_i$ is of order $2B+3$, there must be a partition of $Y$ into triples such that the elements in each triple add up to $2B+3$. Equivalently, there must be a partition of $X$ into triples such that the elements in each triple add up to $B$.
Since $T(X)$ is a tree of maximum degree $3$, then, we have a polynomial time reduction from the Distinct $3$-Partition problem to the Graph Burning problem for trees with maximum degree $3$. 
\end{proof}

Since any tree is a chordal graph, and also planar and bipartite, then we conclude the following corollary.

\begin{corollary}
The burning problem is \textbf{NP}-complete for chordal graphs, planar graphs, and bipartite graphs.
\end{corollary}

If in the proof of Theorem \ref{3tree}, we keep the graphs $Q_i$'s, $Q'_i$'s, and $T_i$'s disjoint, then we will have exactly the same argument to show a polynomial time reduction from the Distinct $3$-Partition problem to the Graph Burning problem. Thus, we have the following immediate corollary as well.

\begin{corollary}
The burning problem is \textbf{NP}-complete for forests of maximum degree three.
\end{corollary}

A \emph{binary tree} is a rooted tree in which every node has at most two children. 
In a \emph{full binary tree}, every node that is not a leaf has exactly two children. 
A \emph{perfect binary} tree is a full binary tree in which all the leaves have the same depth; that is, they are all of the same distance from the root.

Let $T$ be a perfect binary tree of radius $r$. Assume that $s$ is the root of $T$, and $u$ and $v$ are the neighbours of $s$. By definition, we can see that, there is only one node of degree two that is the root $s$, and every other node except for the leaves is of degree three.  Moreover, $s$ is the unique centre node of $T$, and each of $u$ and $v$ by itself is the root of a perfect binary tree of radius $r-1$. Also, note that every perfect binary tree of radius $r$ has $2^r$ leaves.
See Figure \ref{bin} for an example of a perfect binary tree.
\begin{figure}
  \begin{center}
  \includegraphics[width=.25\linewidth]{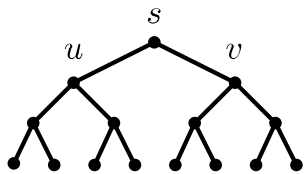}
  \caption{A perfect binary tree with root $s$.}\label{bin}
  \end{center}
\end{figure}

Note that in the proof of Theorem \ref{3tree}, the gadget graph that we construct is a binary tree.
To see this, we can take the end point of path $Q$ that is a leaf as a root for $T(X)$ and then, every other node is a descendant of the root and has at most two children.
Therefore, as another corollary of Theorem \ref{3tree}, we can say that the burning problem is \textbf{NP}-complete even for binary trees.
However, we can find the burning number of the perfect binary trees as stated below.

\begin{theorem}\label{perfbin}
If $T$ is a perfect binary tree of radius $r$, then $b(T) = r+1$.
\end{theorem}

\begin{proof}
We prove the theorem statement by induction on $r$ as follows. First, by definition, a perfect binary tree $T$ of radius $1$ consists of only a root node and its two neighbours. Thus, $T$ is a path on three nodes, and hence, by Theorem \ref{b222}, $b(T) =2$. 

Now, assume that the burning number of a perfect binary tree of radius $r-1$ equals $r$, and $T$ is a perfect binary tree of radius $r$. 
Since the root of a perfect binary tree is its centre, then, by burning the root at the first step, every other node will be burned after $r$ more steps. 
Let $s$ be the root of $T$ with neighbours $u$ and $v$. As we mentioned right after the definition of a perfect binary tree, $u$ and $v$ are the roots of two perfect binary trees of radii $r-1$ that we call them $T_1$ and $T_2$, respectively.
Since $T$ includes a perfect binary tree of radius $r-1$ as its isomorphic subtree, by Corollary \ref{subforest}, we conclude that $r\leq b(T)\leq r+1$. 

By contradiction, suppose that $b(T) = r$. Clearly, $V(T)$ is the disjoint union of  $V(T_1)$,  $V(T_2)$, and  $\{s\}$. Hence, for burning $T$ we have to burn $T_1$ and $T_2$ as well. Without loss of generality, assume that in an optimum burning sequence for $T$ like $(x_1, x_2, \ldots, x_r)$, we have that $x_1 \not\in T_1$. 
We know that $T_1$ has $r$ levels (including $u$) and the only connection of $T_1$ to the rest of the nodes in $T$ is through the node $s$. Hence, the distance between every leaf in $T_1$ and any node in $T\setminus T_1$ is at least $r$. Since $T_1$ has $2^{r-1}$ leaves, and every leaf in $T_1$ must be burned by the end of the $r$-th step, then, we have to choose at least one source of fire from $T_1$. 

Now, let $L$ denote the set of the leaves in $T_1$. Since the number of the leaves in a perfect binary tree of radius $r-i$ equals $2^{r-i}$, then we conclude that for $2\leq i\leq r$, $|N_{r-i}[x_i] \cap L| \leq 2^{r-i}$ (where the equality holds only if we choose $x_i$ from the $(i-1)$-th level of $T_1$).
Hence, by following this greedy argument, even in the case that we choose all the $x_i$'s, for $2\leq i\leq r$, from $T_1$, we will burn at most $2^{r-2} + 2^{r-3} + \cdots + 1 = 2^{r-1} - 1$ leaves in $L$, which is a contradiction. By symmetry of $T$, we have a similar argument for not choosing $x_1$ from $T_2$. Hence, we conclude that burning $T$ in $r$ steps is impossible. Therefore, $b(T) = r+1$, and the proof follows by induction. 
\end{proof}

\section{Burning Spider Graphs and Path-Forests}

In this section, we prove that the Graph Burning problem is \textbf{NP}-complete even for spider graphs and path-forests.
We first provide some background on the burning number of trees.

A \emph{terminal path} in a tree $T$ is a path $P$ in $T$ such that one of the end points of $P$ is a leaf of $T$. The other end point of $P$, that is not necessary a leaf, is called the \emph{non-terminal} end point of $P$ (if $P$ is of order one, then the non-terminal end point of $P$ and the leaf in $P$ coincide). Assume that $\{Q_i\}_{i=1}^t$ is a set of disjoint terminal paths  in $T$, and let $v_i$ denote the non-terminal point of the path $Q_i$, for $1\leq i\leq t$. We call $\{Q_i\}_{i=1}^t$ a \emph{decomposed spider} in $T$ if the path between every pair $v_i$ and $v_j$ does not contain any node of $Q_i$ and $Q_j$ except $v_i$ and $v_j$. 

\begin{theorem}\label{lowkt}
Suppose that $\{Q_i\}_{i=1}^t$, where $t\geq 3$, forms a decomposed spider in a tree $T$, and let $v_i$ be the non-terminal end point of $Q_i$, for $1\leq i\leq t$. If $d(v_i , v_j)\geq 2k$ for all $1\leq i,j\leq t$, and $t\geq k$, then $b(T) \geq k+1$.
\end{theorem}

\begin{proof}
Let $T'$ be the smallest connected subgraph of $T$ that contains $\bigcup_{i=1}^t Q_i$.
Since $T'$ is an isometric subtree of $T$, to prove that $b(T)\geq k+1$ it suffices to show that $b(T') \geq k+1$, as follows.

First, we show that the burning number of $T'$ is at least $k$. Let $w_i$ denote the leaf of $T'$ in $Q_i$.
Note that we may have $w_i = v_i$ (in the case that $Q_i$ is of order one).
We claim that there is no fire source $x_j$ that spreads the fire to two distinct leaves. By contradiction, suppose that there are two distinct leaves $w_i$ and $w_r$, and a fire source $x_j$ for which we have that $d(x_j , w_i) \leq k-j$ and $d(x_j , w_r) \leq k-j$ (that is, $w_i$ and $w_r$ both receive the fire started at $x_j$).
By triangle inequality, we conclude that 
$$2k\leq d(v_i , v_r) \leq d(w_i ,w_r) \leq d(w_i , x_j) + d(x_j, w_r) \leq 2k-2j < 2k,$$ which is a contradiction.
Therefore, it implies that corresponding to every leaf $w_i$ there is a unique fire source $x_{j}$ such that the fire spread from $x_{j}$ only burns one leaf of $T$, that is $w_i$.
Thus, the number of fire sources must be at least as large as the number of the leaves in $T'$ that is $t\geq k$.
Hence, we must have $b(T') \geq k$.

Now, we claim that $b(T') \neq k$. 
By contradiction, suppose that $b(T') = k$, and $(x_1, x_2, \ldots, x_k)$ is an optimum burning sequence for $T'$.
If $t > k$, then the above argument leads to the same contradiction, as the number of the fire sources has to be as large as the number of the leaves. If $t = k$, then let $w_i$ be the leaf that receives the fire from $x_k$. Since $b(T') = k$, then it implies that $x_k = w_i$.
We claim that there is no fire source $x_j \neq x_k$ with $d(v_i , x_j) \leq k$. By contradiction, suppose that there is a fire source $x_j \neq x_k$ with $d(v_i , x_j) \leq k$, and let $w_{r}$ be the leaf of $T'$ that receives the fire spread from $x_j$. Thus, we have that
$$2k \leq d(v_{r} , v_i) \leq d(w_{r} , v_i) \leq  d( w_{r}, x_j) + d( x_j, v_i ) \leq k-j + k < 2k,$$
which is a contradiction.

Let $s$ be a neighbour of $v_i$ that is not in the path between $v_i$ and $w_i = x_k$. Since $t\geq 3$,  we are sure that such a node $s$ does exist. By assumption, we know that $x_k = w_i$, and therefore $s$ can not receive the fire spread from $x_k$. On the other hand, the distance between $s$ and any other fire source must be at least $k$. Thus, $s$ can not be burned by the end of the $k$-th step, which is a contradiction.
Hence, we have that $b(T') \geq k+1$.
\end{proof}

Assume that we want to find the burning number and an optimum burning sequence for a given tree $T$. If there is an optimum burning sequence $(x_1, x_2, \ldots, x_k)$  for $T$ such that $V(T) \subseteq N_{k-1}[x_1]$, then we can see that $b(T) = \mathrm{radius}(T) +1$. If $T \setminus N_{k-1}[x_1]$ is non-empty, then it implies that $$T\setminus N_{k-1}[x_1] \subseteq N_{k-2}[x_2] \cup N_{k-3}[x_3] \cup \ldots \cup N_0[x_k].$$ 
This observation suggests the following conjecture.

\begin{conjecture}\label{lowk}
Suppose that $\{Q_i\}_{i=1}^t$, where $t\geq 3$, forms a decomposed spider in a tree $T$, and let $v_i$ be the non-terminal end point of $Q_i$, for $1\leq i\leq t$. If $b(\cup_{i=1}^t Q_i) \geq k$, and $d(v_i , v_j)\geq 2k$ for all $1\leq i,j\leq t$, then $b(T) \geq k+1$.
\end{conjecture}

Conjecture \ref{lowk} may be helpful in finding a lower bound on the burning number of a tree $T$ (as we can see the truth of it for paths by Theorem \ref{path} \cite{bonato2}, and also we will see the truth of it later on for some specific spider graphs). In particular, if the burning number of a tree $T$ is strictly less than $\mathrm{radius}(T) +1$, and the conjecture was true, then by starting from the leaves of $T$ we could find a good lower bound on $b(T)$. 
Note that by Theorem \ref{lowkt}, when $t\geq k$, the above conjecture is true. Also, we can prove the following lemma, since the leaves in any spider graph $SP(s,r)$, with $s\geq r$, form a decomposed spider that satisfies the conditions in Theorem \ref{lowkt}. 

\begin{lemma}\label{sk}
For a spider graph $SP(s,r)$, with $s\geq r$, we have that $b(SP(s,r))= r+1$. Moreover, for $s\geq r+2$, every optimal burning sequence of $SP(s,r)$ must start by burning the central node.
\end{lemma}

\begin{proof}
We know that $b(SP(s,r))\leq r+1$, as $SP(s,r)$ has radius $r$. Since $SP(r,r)$ is an isometric subgraph of $SP(s,r)$ where $s\geq r$, then it suffices to show that $b(SP(r,r)) = r+1$.

First, we prove that $b(SP(r,r)) \geq r+1$. 
We index the leaves of $SP(r,r)$ with $w_1, w_2, \ldots, w_r$.
For $1\leq i\leq r$, let $Q_i$ be the graph induced by $w_i$; that is, $Q_i$ is a path of order one.
Hence, every $Q_i$ is a terminal path in $SP(r,r)$ with the non-terminal end $w_i$, and for every distinct pair $1\leq i,j\leq r$, we have that $d(w_i ,w_j) = 2r$.  Therefore, by Theorem \ref{lowkt}, we conclude that $b(SP(r,r)) \geq r+1$. Hence, the proof follows.

Now, suppose that $s\geq r+2$ and there exists an optimal burning sequence $(x_1, x_2, \ldots, x_{r+1})$ for $SP(s,r)$ in which $x_1$ is not the central node. Since $s\geq r+2$ and $b(SP(s,r)) = r+1$, then by Pigeonhole Principle, one of the arms does not include any source of fire, unless we choose the central node as a fire source. Note that by assumption, $x_1$ is not the central node. Since the only connection between the nodes in that arm to the rest of the nodes in $SP(s,r)$ goes through the central node, then in both cases, we need at least $1 + (r+1)$ steps for burning the leaf on that arm, which is a contradiction. Thus, every optimal burning sequence of $SP(s,r)$ starts by burning the central node where $s\geq r+2$.
\end{proof}

Using the above lemma, we now prove that the burning problem is \textbf{NP}-complete even for spider graphs.
\begin{theorem}\label{trees}
The burning problem is \textbf{NP}-complete for spider graphs.
\end{theorem}

\begin{proof}
Clearly, by Theorem \ref{3tree}, the burning problem is in \textbf{NP}. As in the proof of Theorem \ref{3tree}, we give a reduction from the Distinct $3$-Partition problem into the burning problem, in which the gadget graph that we construct is a spider graph.

Given an instance of the Distinct $3$-Partition problem, that is, a set $X =\{a_1, a_2, \ldots, a_{3n}\}$ of positive distinct integers and a positive integer $B$ such that each $B/4 < a_i < B/2$, we construct a graph $G$ as follows. 
Since the Distinct $3$-Partition problem is strongly \textbf{NP}-complete (as in the proof of Theorem \ref{3tree}), without loss of generality we assume that $B$ is bounded above by a polynomial in terms of $n$.

Suppose that $\max{X} = m+1$, and let $Y =\{2a_i -1 : a_i \in X\}$. Clearly, $Y\subseteq O_{m+1}$. Then we make a copy of the spider graph $SP(2m+5, m+1)$ with centre $s$, called $G_s$. Now, for any positive odd integer $l\in O_{m+1} \setminus Y$, we connect a leaf of copy of $P_l$ (a path on $l$ nodes) to a distinct leaf of $SP(2m+5, m+1)$. We attach $n$ copies of $P_{2B+3}$, called $Q_1, Q_2, \ldots, Q_n$ to distinct leaves of $SP(2m+5, m+1)$ that we have not used for attaching any other $P_l$, with $l \in O_{m+1} \setminus Y$. We call the resulting graph $G$. Clearly, $G$ is a spider tree with spider head $s$. Since $V(G)$ is the disjoint union of the spider graph $SP(2m+5, m+1)$ and the paths $Q_1, Q_2, \ldots, Q_n$, and the paths $P_l$, with $l\in O_{m+1} \setminus Y$, we have that
$$|V(G)| = \sum_{i=1}^{m+1} (2i-1) + (2m+5)(m+1) + 1 = (m+1)^2 + (2m+5)(m+1) +1,$$ which is of order $O(B^2)$ in terms of $B$. Since $B$ is bounded above by a polynomial in terms of $n$, it implies that we construct graph $G$ in polynomial time in terms of $n$. We want to show that, there is a partition of $X$ into $n$ triples such that the numbers in each triple add up to $B$ if and only if $b(G)\leq m+2$.

First, assume that there is a partition of $X$ into $n$ triples such that the numbers in each triple add up to $B$. Consequently, paths $Q_1, Q_2, \ldots, Q_n$ can be partitioned into smaller paths of orders $\{2a_i-1 : a_i \in X\}$. For $l \in Y$, we set $x_{m+2 - (\frac{l -1}{2})}$ to be the middle node of the paths $P_{l}$, applied in such a partition of $Q_1, Q_2, \ldots, Q_n$. Then we take $x_1 = s$, and for any $l\in O_{m+1} \setminus Y$, we set the middle node of $P_l$ as $x_{m+2 - (\frac{l-1}{2})}$. The sequence $(x_1, x_2, \ldots , x_{m+2})$ is a burning sequence for $G$. Thus, $b(G)\leq m+2$. 

For example, let
$X = \{10, 11, 12, 14, 15, 16\}$, and $B = 39$. Then the graph $G$ is shown in Figure \ref{tree}. Here, we have that $n=2$, and $m = \max\{a_i: a_i \in X\} - 1 = 15 $. Therefore, $Y = \{19, 21, 23, 27, 29, 31\}$, and $O_{16} \setminus Y = \{1,3,5,7,9, 11, 13, 15, 17, 25\}$. The red nodes in Figure \ref{tree} denote a burning sequence of length $17$ for tree $G$. 

\begin{figure}[]
\begin{center}
\includegraphics[scale=0.7]{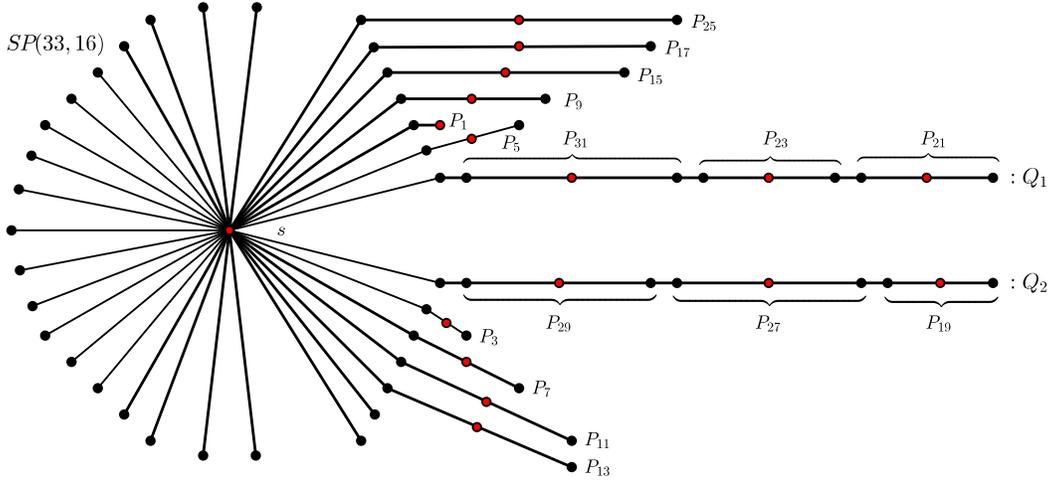}
\caption{A sketch of the tree $G$.} \label{tree}
\end{center}
\end{figure}

Conversely, suppose that $b(G)\leq m+2$. Since $G_s$ is an isometric subtree of $G$, then, Theorem \ref{subforest} and Lemma \ref{sk} imply that $b(G) = m+2$. Thus, $G$ has a burning sequence $(x_1, x_2, \ldots, x_{m+2})$. We have to show that there is a partition of $X$ into $n$ triples such that the numbers in each triple add up to $B$.  First, note that we use at most $m+1$ leaves of $G_s$  for attaching the paths $P_l$, with $l \in O_{m+1}$, and the paths $Q_1, Q_2, \ldots, Q_n$. Thus, there is a copy of $SP(m+4, m+1)$ that is an isometric subtree of $G$ and the only connection of its leaves to the rest of $G$ is through node $s$. Therefore, by Lemma \ref{sk}, we conclude that $x_1 = s$. On the other hand, by burning node $s$ at the first step, all the nodes in $G_s$ will be burned by the end of the $(m+2)$-th step. Thus, without loss of generality we can assume that for $2\leq i\leq m+2$, all $x_i$'s are selected from $G \setminus G_s$.

Now, by equation (\ref{eqq}), we know that $G \setminus G_s = \bigcup_{i=2}^{m+2}N_{m+2-i}[x_i]$. Since $G \setminus G_s$ is a path-forest, then $N_{m+2-i}[x_i]$ must be a path of order at most $l = 2(m+2 -i) +1$, for $2\leq i\leq m+2$. Besides, we have that 
\begin{align*}
|V(G \setminus G_s )|& = 2nB + 3n + \sum_{l\in O_{m+1} \setminus Y} l \\
& = \sum_{i=2}^{m+2} \left(2(m+2 -i) +1\right).
\end{align*}
Therefore, it implies that $N_{m+2-i}[x_i]$ must be a path of order exactly equal to $l = 2(m+2 -i) +1$, for $2\leq i\leq m+2$. 
Hence, there must be a partition of $G \setminus G_s$ by the set of paths of orders in $O_{m+1}$, in which the center of each path in the partition is a fire source.

We claim that there is a burning sequence for $G$ in which the central node of each $P_l$, $l\in O_{m+1} \setminus Y$ (that we attached to a leaf of $G_s$), is selected as a fire source. We can easily prove this claim by switching the paths that are possibly displaced in the current partition for $G \setminus G_s$.
Thus, the closed neighbourhoods of the rest of the fire sources form a partition for $Q_1, Q_2, \ldots, Q_n$ in terms of paths of orders $2a_i -1\in Y$. Since each $Q_i$ is of order $2B+3$, then it implies that there is partition for $ X$ into triples such that the elements in each triple  add up to $B$. 
\end{proof}

If we delete the spider graph $SP(2m+5, m+1)$ in the proof of Theorem \ref{trees}, and keep the rest of the parts of the gadget graph $G$ the same, then we will have the analogous argument for the disjoint union of the paths $Q_1, Q_2, \ldots, Q_n$, and the paths $P_l$ with $l\in O_{m+1} \setminus Y$. Thus, we can have a reduction from the Distinct $3$-Partition problem to the burning problem for the path-forests. Therefore, we conclude the following corollary.

\begin{corollary}
The Burning problem is \textbf{NP}-complete for path-forests.
\end{corollary}

Note that in Theorem \ref{trees} and the above corollary, we do not have any restriction on the number of the arms in $SP(2m+5, m+1)$ and on the length of the paths in constructing the gadget graphs. In other words, the parameter $m$ is unbounded. 

\section{Algorithms for Burning Path-Forests and Spider Graphs}

In this final section, we present a polynomial time algorithm that finds the burning number of path-forests when the number of components and their orders are restricted, and then we find another polynomial time algorithm that finds the burning number of spider trees with fixed number of arms and with restrictions on the length of the arms.
We first provide some terminology.

Let $G$ be a path-forest with components $Q_1, Q_2, \ldots, Q_k$, where $k\geq 1$, and the order of each path $Q_i$ is $l_i$ such that $l_1\geq l_2 \geq \cdots \geq  l_k$. In other words, we assume that the paths are indexed according to the decreasing order of their lengths. We say that $G$ is a \emph{maximal path-forest} if it can be decomposed into paths of orders $1, 3, \ldots, 2t-1$ for some positive integer $t$.
It is clear that such a graph $G$ is of order $t^2$.
Let $\mathrm{MPF}_t$ denote the set of all maximal path-forests of order $t^2$. If $G$ is a path-forest with burning number $t$, then $G$ corresponds to a sequence of positive integers such as $(l_1, l_2, \ldots, l_s)$, where $s\leq t$, and $l_1 \geq l_2\geq  \cdots \geq l_s$, in which $l_i$ denotes the order of the $i$-th component of $G$. 
From now on, we represent a path-forest with burning number $t$ by a sequence of integers as defined above.

We denote the set of maximal path-forests with $t$ components and with burning number $k$ by $\mathrm{MPF}^t_k$. For example, $\mathrm{MPF}^1_1 = \{P_1\} = \{(1)\}$.
In general, we can see that for any $k\geq 1$, $\mathrm{MPF}^k_k = \{(2k-1, 2k-3, \ldots, 1)\}$. 
Also, note that for any $k\geq 1$, $\mathrm{MPF}^1_k = \{P_{k^2}\}$.

\begin{algorithm}\label{alg1}
Suppose that $G = (s_1,s_2, \ldots, s_t)$, for a constant $t\geq 1$, represents a path-forest in which $s_i$ denotes the order of the $i$-th component of $G$, and $s_1\geq s_2 \geq \cdots \geq  s_t$. Let $m$ be a positive integer such that $s_1 \leq m$; that is, the order of the components of $G$ is bounded above by $m$. Then we perform the following steps.

\medskip

{\bf Stage 1.} 
First, for each $1\leq r \leq t-1$, we perform Stages 1.1 and 1.2:

{\bf Stage 1.1.}
We set $\mathrm{MPF}^r_r= \{(2r-1, 2r-3, \ldots, 1)\}$.

If $(s_1, s_2, \ldots, s_r) \not\in \mathrm{MPF}^r_r$, then go to the next step.

\medskip

{\bf Stage 1.2.} For $k\geq r+1$, we perform the following steps:

\medskip

{\bf Stage 1.2.1.} For each $H = (l_1, l_2, \ldots, l_{r-1}) \in \mathrm{MPF}^{r-1}_{k-1}$, we form the sequence $H' = (2k-1, l_1 , \ldots , l_{r-1})$. We rearrange the numbers in the sequence $H'$ if they do not appear in a decreasing order, and we add it to the set $\mathrm{MPF}^{r}_k$.

If $(s_1, s_2, \ldots, s_r) \subseteq H'$, then finish Stage 1.2, and go to the next stage.

\medskip

{\bf Stage 1.2.2.} For each $H = (l_1, l_2, \ldots, l_{r}) \in \mathrm{MPF}^{r}_{k-1}$, and each $1\leq i\leq r$, we form the sequence $H_i = (l_1 , \ldots ,l_{i-1}, l_i + 2k-1, l_{i+1}, \ldots, l_{r})$.
We rearrange the numbers in the sequences $H_i$ if they do not appear in a decreasing order, and we add them to the set $\mathrm{MPF}^r_k$.

If $(s_1, s_2, \ldots, s_r) \subseteq H_i$, then finish Stage 1.2, and go to the next stage.

\medskip

{\bf Stage 2.} For $r = t$, we perform the following steps:

\medskip

{\bf Stage 2.1.} We set $\mathrm{MPF}^t_t= \{(2t-1, 2t-3, \ldots, 1)\}$.

If $G \in \mathrm{MPF}^t_t$, then stop and return $b(G) = t$.

\medskip

{\bf Stage 2.2.} For $k\geq t+1$, we perform the following steps:

\medskip

{\bf Stage 2.2.1.} For each $H = (l_1, l_2, \ldots, l_{t-1}) \in \mathrm{MPF}^{t-1}_{k-1}$, we form the sequence $H' = (2k-1, l_1 , \ldots , l_{t-1})$. 
We rearrange the numbers in the sequence $H'$ if they do not appear in a decreasing order, and we add it to the set $\mathrm{MPF}^{t}_k$.

If $G \subseteq H'$, then stop and return $b(G) = k$.

\medskip

{\bf Stage 2.2.2.} For each $H = (l_1, l_2, \ldots, l_{t}) \in \mathrm{MPF}^{t}_{k-1}$, and for each $1\leq i\leq t$, we form the sequence $H_i = (l_1 , \ldots ,l_{i-1}, l_i + 2k-1, l_{i+1}, \ldots, l_{t})$.
We rearrange the numbers in the sequence $H_i$ if they do not appear in a decreasing order, and we add it to the set $\mathrm{MPF}^t_k$.

If $G \subseteq H_i$, then stop and return $b(G) = k$.

\end{algorithm}

The algorithm works since every graph $G$ that is not a subgraph of a graph in $\mathrm{MPF}^t_i$, for all $1\leq i< k$, but $G$ is a subgraph of a graph in $\mathrm{MPF}^t_k$, has burning number $k$. 
Besides, we have the following fact about Algorithm~\ref{alg1}.

\begin{theorem}\label{polyalg1}
Suppose that $G = (s_1,s_2, \ldots, s_t)$, for an integer constant $t\geq 1$, represents a path-forest in which $s_i$ denotes the order of the $i$-th component of $G$, and $s_1\geq s_2 \geq \cdots \geq  s_t$. Let $m$ be a positive integer such that $s_1 \leq m$; that is, the order of the components of $G$ is bounded above by $m$. If $t $ is a fixed constant in terms of $m$, then Algorithm \ref{alg1} finds the burning number of $G$ in polynomial time in terms of the input.
\end{theorem}

\begin{proof}
Given the graph $G$, suppose that for some $k\geq t$, Algorithm \ref{alg1} stops by recognizing $G$ as a subgraph of a graph in $\mathrm{MPF}^t_k$. Note that $t$ is a fixed constant in terms of $m$. Thus, by Theorem \ref{path-forest}, we derive that if $H = (l_1, l_2, \ldots, l_r)$ is a graph in $\mathrm{MPF}^r_i$ (generated by Algorithm \ref{alg1}), with $1\leq r\leq t$ and $i\geq r$, then
$$b(H) \leq \lceil \sqrt{\sum_{j=1}^r l_j} \rceil + i-1 \leq \sqrt{mt} + t-1 = O(\sqrt{m}).$$
On the other hand, since $b(H) = i$, then there is a partition of the set $O_i$ into subsets $\{A_j\}_{j=1}^r$ such that $l_j = \sum_{a\in A_j} a$, for $1\leq j\leq r$. It implies that $l_j \leq \sum_{a\in O_i} a = i^2 =O(m)$, for $1\leq j\leq r$.
Hence, the length of the longest $l_j$ that appears in the representation of such a graph $H$ is of order $m$. Let $l = O(m)$ be the length of the longest component in a graph $H$ generated by Algorithm \ref{alg1}. Thus, any graph $H$ generated by Algorithm \ref{alg1} is a subgraph of the graph $G_0 = (l, l, \ldots, l)$ with $t$ components. Since these graphs are distinct, then the total number of graphs generated by Algorithm \ref{alg1} is of order $O(m^t)$. 

Moreover, note that for $r=t$ and $k\geq t$, each time that we add a new graph $H = (l_1, l_2, \ldots, l_t)$ to $\mathrm{MPF}^t_k$, we check to see if $G$ is a subgraph of $H$ or not. We simply can do this comparison by checking if $s_i \leq l_i$, for $1\leq i\leq t$. Thus, the total number of steps that we perform in Algorithm \ref{alg1} is bounded above by $O(tm^t)$.
Since $t$ is a fixed constant in terms of $m$, then Algorithm \ref{alg1} is a polynomial time algorithm in terms of the input.
\end{proof}

In the following, we try to find the burning number of spider graphs, again using a dynamical programming approach. 
First we need some facts to use for this algorithm. 
We state the following lemma since we use it for proving the next theorem.
Assume that $G$ and $H$ are two disjoint graphs, and $u\in G$ and $v\in H$ are two nodes. We can make a new graph $G+uv+H$ by adding edge $uv$ to $G\cupdot H$.

\begin{lemma}\label{glue}
If $G$ and $H$ are two disjoint non-empty graphs then we have that
$$b(G+uv+H)\leq b(G\cupdot H),$$
where $u\in V(G)$ and $v\in V(H)$.
\end{lemma}

\begin{proof}
Since $V(G+uv+H)= V(G\cupdot H)$, then every burning sequence for $G\cupdot H$ induces a covering for $G+uv+H$; in particular, any minimum burning sequence of $G\cupdot H$ induces a covering for $G+uv+H$. Therefore, $b(G+uv+H)\leq b(G\cupdot H)$.
\end{proof}

The following theorem plays a key role in the algorithm that we will present, and shows that for a spider tree we always can have an optimum burning sequence in which the first source of fire is close to the spider head.
 
\begin{theorem}\label{key}
If $G$ is a spider graph with $s\geq 3$ arms and the spider head $c$, then there is an optimum burning sequence $(x_1, x_2, \ldots, x_k)$ for $G$ such that $d(x_1, s) \leq k-1$.
\end{theorem}

\begin{proof}
We prove this by strong induction on the number of the nodes in $G$.
The smallest order spider graph is a star with three leaves.
By Theorem \ref{b222},  we know that the burning number of such a star equals $2$ and in every optimum burning sequence for this graph the first fire must be the centre that is the spider head. Hence, the theorem statement is true for this spider.

Now, suppose that the theorem statement is true for every spider graph of order at most $n-1$, and $G$ is a spider graph of order $n$ with $s\geq 3$ arms and spider head $c$. Also, assume that $L_1, L_2, \ldots, L_s$ are the arms of $G$, and $v_1, v_2, \ldots, v_s$ are their corresponding leaves. Finally, suppose that the order of each arm $L_i$ is denoted by $l_i$.
Let $(x_1, x_2, \ldots, x_k)$ be an optimum burning sequence for $G$. By equation~(\ref{eqq}), we know that
$$V(G) = N_{k-1}[x_1] \cup N_{k-2}[x_2] \cup \ldots \cup N_0[x_k].$$
If $d(x_1, c) \leq k-1$, then we are done. Hence, let $d(x_1, c) \geq k$, and $x_1 \in L_i$ where $1\leq i\leq s$.
We consider two possibilities for $l_i$: either $l_i\leq 2k-2$ or $l_i \geq 2k-1$.

{\bf Case 1.} If $l_i \leq 2k-2$, then it implies that $d(c,v_i) \leq 2k-2$. Let $x$ be the node in $L_i$ for which $d(x,v_i) = k-1$. Therefore, we have that $d(c, x) \leq k-1$. Note that we can cover all the nodes in $L_i \cup \{c\}$ with $N_{k-1}[x]$. Hence, $V(G) \setminus N_{k-1}[x] \subseteq V(G) \setminus N_{k-1}[x_1]$. Thus, we still have that 
$$V(G) = N_{k-1}[x] \cup N_{k-2}[x_2] \cup \ldots \cup N_0[x_k].$$
Note that some of the fire sources $x_j$'s, with $j\geq 2$, might be in $N_{k-1}[x] \cap L_i$. 
Therefore, we have that $b(G \setminus N_{k-1}[x]) = t \leq k-1$. Hence, we can find a burning sequence of length $t$ such as $(x'_2, x'_3, \ldots, x'_t)$ for $G \setminus N_{k-1}[x]$. Also, for $t +1 \leq j \leq k$, we define $x'_j$ to be a node of distance $j-1$ from $x$. Thus, for $t+1 \leq j \leq k$, $d(x'_j , x) \geq j-1$, and 
$d(x'_j , x_r) \geq r - 1 + j-1 \geq j-r$, for any $2\leq r \leq t$. Therefore, the sequence $(x'_1 = x, x'_2, \ldots, x'_k)$ forms a desired optimum burning sequence for $G$.

{\bf Case 2.} If $l_i \geq 2k-1$, then either $d(v_i , x_1) \leq k-1$ or $d(v_i, x_1) \geq k$.
We claim that there is a burning sequence for $G$ such as $(x'_1, x'_2, \ldots , x'_k)$ such that $x'_1 \in L_i$ and $d(x'_1, v_i) \leq k-1$, or equivalently, $G \setminus N_{k-1}[x'_1]$ is connected. If $d(x_1, v_i) \leq k-1$, then we are done. If  $d(v_i , x_1) \geq k$, then 
$G \setminus N_{k-1}[x_1]$ is the disjoint union of a spider graph $G'$ and a path $P$, such that $P$ is a subpath of $L_i$ containing $v_i$. Let $u$ be the leaf of $G'$ that is in $L_i$, and $v$ be the other end point of $P$ that probably is different from $v_i$. We know that $b(G \setminus N_{k-1}[x_1])\leq k-1$. Hence, by Lemma \ref{glue}, we have that 
$$t = b(G' +uv+  P) \leq b(G' \cup P) = b(G \setminus N_{k-1}[x_1]) \leq k-1.$$
Note that $G' +uv+ P$ is a subtree of $G$ that is (isomorphic to) a spider of the same number of arms as $G$. In fact, the $i$-th arm of $G' +uv+ P$ is (isomorphic to) a subpath of $L_i$ with exactly $2k-1$ less nodes than $L_i$. Also, note that some of the fire sources $x_j$'s, with $j\geq 2$, might be in $N_{k-1}[x_1] \cap L_i$. Let $(x'_2, x'_3, \ldots, x'_t)$ be an optimum burning sequence for $G' +uv+ P$, and $x'_1$ be the node in $L_i$ with $d(x'_1 , v_i) = k-1$. Also, for $t+1 \leq j \leq k$, we take $x'_j$ to be a node of distance $j$ from $x'_1$ that is on the path connecting $x'_1$ and $v_i$.
Thus, the sequence $(x'_1, x'_2, \ldots, x'_k)$ forms a burning sequence for $G$, such that $x'_1 \in L_i$, and $G \setminus N_{k-1}[x'_1]$ is connected.

Now, by above claim, without loss of generality, we assume that $N_{k-1}[x_1]$ contains $v_i$. That is, we have a burning sequence $(x_1, x_2, \ldots, x_k)$ for $G$ such that $G' = G \setminus N_{k-1}[x_1]$ is a spider graph with smaller number of nodes than $G$, and with the same number of arms and the same spider head $c$. In fact, for $j\neq i$, and $1\leq j\leq s$, $L_j$ is the $j$-th arm of $G'$ too, and the $i$-th arm of $G'$ is a subset of $L_i$ that contains exactly $2k-1$ nodes less than $L_i$. 
Hence, we have that $b(G')= t \leq k-1$, and by induction hypothesis, $G'$ must have a burning sequence $(x'_2, x'_3, \ldots, x'_t)$ such that $d(x'_2, c) \leq t-1 \leq k-2$. We have two possibilities: either $x'_2 \in L_i$, or $x'_2 \in L_j$ for some $j\neq i$.  

If $j = i$, then let $x$ be the neighbour of $x'_2$ that is on the path which connects $x'_2$ to $v_i$. Also, let $x'$ be the neighbour of $x_1$ that is on the path connecting $x_1$ to $v_i$.
Hence, we have that $$G \setminus (N_{k-1}[x] \cup N_{k-2}[x']) = G \setminus (N_{k-1}[x_1] \cup N_{k-2}[x'_2]).$$
Now, for $t+1 \leq r \leq k$, we take $x'_r$ to be the node in $L_i$ on the path connecting $v_i$ to $x'$ that is of distance $r -2$ from $x'$. Finally, we take $x'_1 = x$, and we redefine $x'_2 = x'$. Thus, the sequence $(x'_1, x'_2, \ldots, x'_k)$ forms a burning sequence for $G$ in which $d(x'_1, c) \leq k-1$. 

If $j\neq i$, then let $x$ be the neighbour of $x'_2$ that is on the path connecting $x'_2$ to $c$. Also, let $x'$ be the neighbour of $x_1$ that is closer to $v_i$.
Hence, we have that 
$$L_i \setminus (N_{k-1}[x] \cup N_{k-2}[x']) = L_i \setminus (N_{k-1}[x_1] \cup N_{k-2}[x'_2]),$$ 
(by isomorphism). 
Also, $$L_j \setminus (N_{k-1}[x] \cup N_{k-2}[x']) = L_j \setminus (N_{k-1}[x_1] \cup N_{k-2}[x'_2]).$$ 
But, $$G \setminus (N_{k-1}[x] \cup N_{k-2}[x']) \subseteq G \setminus (N_{k-1}[x_1] \cup N_{k-2}[x'_2]),$$ 
and we know that $N_{t-2}[x'_3] \cup N_{t-3}[x'_4] \cup \ldots \cup N_0[x'_t]$ forms a covering for $G \setminus (N_{k-1}[x] \cup N_{k-2}[x'])$. In fact, $G \setminus (N_{k-1}[x] \cup N_{k-2}[x'])$ is an isometric subforest of $G \setminus (N_{k-1}[x_1] \cup N_{k-2}[x'_2])$.
Thus, by Corollary \ref{subforest}, we have that 

\begin{align*}
b(G \setminus (N_{k-1}[x] \cup N_{k-2}[x'])) &\leq b(G \setminus (N_{k-1}[x_1] \cup N_{k-2}[x'_2])) \\
& \leq t-1 \leq k-2.
\end{align*} 

Hence, there must be an optimum burning sequence $(x''_3, x''_4, \ldots, x''_{t'})$, where $t'\leq t$ for $G \setminus (N_{k-1}[x] \cup N_{k-2}[x'])$.
Now, for $t'+1 \leq r \leq k$, we take $x''_r$ to be the node in $L_i$ on the path connecting $v_i$ to $x'$ that is of distance $r -2$ from $x'$. Finally, we take $x''_1 = x$, and we define $x''_2 = x'$. Thus, the sequence $(x''_1, x''_2, \ldots, x''_k)$ forms a burning sequence for $G$ in which $d(x''_1, c) \leq k-1$. 
\end{proof}

The following lemma provides us with another key tool for finding the burning number of spider graphs.

\begin{lemma}\label{key2}
Let $G$ be a spider graph with spider head $c$. 
Also, suppose that for a positive integer $k$ and a node $x\neq c$ in $G$, $G\setminus N_{k-1}[x]$ is a path-forest (that is, $d(x,c) \leq k-1$)  with at least two components, and  $b(G \setminus N_{k-1}[x]) \leq k-1$. If $x\neq c$, and the neighbour of $x$ on the path connecting $x$ to $c$ is $x'$, then we have that $b(G \setminus N_{k-1}[x']) \leq k-1$. 
\end{lemma}

\begin{proof}
Assume that a spider graph $G$ with the above conditions is given, and we have the nodes $x$ and $x'$ as mentioned in the lemma's statement. Let $x$ be in an arm of $G$ called $L_
s$. We have two possibilities for $L_s$: either $L_s \setminus N_{k-1}[x]$ is empty or not.

First, suppose that $L_s \setminus N_{k-1}[x]$ is not empty. Hence, since in this case one of the components of $G\setminus N_{k-1}[x]$ is contained in $L_s$, then it implies that $L_s$ is of order at least $k + d(x,c)$.
By assumption, we know that each component of $G_1 = G\setminus N_{k-1}[x]$ is a subset of one of the arms in $G$. Let $G_2$ be the path-forest $G \setminus N_{k-1}[x']$. We know that $G_2$ is a path-forest since by assumption, $d(x',c) \leq d(x,c)\leq k-1$.
Hence, each component of $G_2$ is also a subpath of an arm in $G$.
We call the components of $G_1$ and $G_2$ that are subpaths of $L_s$ by $P$ and $P'$, respectively.
In fact, $P'$ is a superset of $P$ with exactly one more node. Also, each non-empty component of $G_2$ such as $Q' \neq P'$ is a subset of the corresponding component $Q \neq P$ of $G_1$, and has exactly one node less than $Q$. 

Since, by assumption, $b(G_1) = t \leq k-1$, then there must be a burning sequence $(x_1, x_2, \ldots, x_{t})$ for $G_1$. Note that each $N_{t-j}[x_j]$ is a path of order at most $2(t-j) +1$. 
Therefore, the closed neighbourhoods of the $x_i$'s cover all the nodes in $G_2$, except for probably the extra node in $P'$ that is a superset of $P$. We have two possibilities: either there is a component $Q \neq P$ in $G_1$ that is of order one, or the order of each component of $G_1$ is of order at least two.

If there is a component $Q\neq P$ of $G_1$ that is of order one and is burned by $x_{i}$, then let $x'_{i}$ be the extra node in $P' \setminus P$. 
If $(x_1, x_2, \ldots, x_{t})$ does not burn $x'_i$, then the sequence $(x_1, \ldots, x_{i-1}, x'_i, x_{i+1}, \ldots, x_{t})$ is a burning sequence for $G_2$. Thus, $b(G_2) \leq k-1$.

If every component of $G_1$ like $Q \neq P$ is of order at least two, then we have again two possibilities: either $x_{t}$ is in $P$, or $x_{t} \not\in P$.

If $x_{t}$ is in $P$, then let $i$ be the smallest index for which $x_i \in P$, but $x_{i-1}$ is not in $P$.
We know that such an index $i$ does exist, since otherwise, it means that all the $x_i$'s must be in $P$, and consequently, it implies that $P$ is the only component of $G_1$, which is a contradiction. 
Thus, there must an index $i$ such that $x_i \in P$, but $x_{i-1}$ is in a component of $G_1$ that we call it $Q$, with $Q \neq P$.

Since, each $N_{t-j}[x_j]$ is a path of order at most $2(t-j) +1$, then without loss of generality we can assume that $N_{t-i}[x_i]$ covers at least two nodes less than $N_{t-(i-1)}[x_{i-1}]$.
Now, let $x'_i = x_{i-1}$ and $x'_{i-1} = x_i$. Therefore, we have a new covering for $G_2$ induced by $(x_1, \ldots, x_{i-2}, x'_{i-1} , x'_{i}, \ldots , x_{t})$ in which all the nodes of $P'$ plus one extra node of $L_s \setminus P'$ is covered, while we may have lost covering one node in $Q$. Now, by moving $x_{t}$ to cover such a uncovered node in $Q$, and shifting the place of the fire sources used for covering $P$ without changing their order (if it is necessary), we find a covering for $G_2$ with $t$ closed neighbourhoods of restricted radii. Hence, by Corollary \ref{cov3}, $b(G_2) \leq k-1$. 

If $x_{t}\not\in P$, then there must be a component $Q$ of $G_1$ for which $x_{t}\in Q$.
Let $Q'$ be the corresponding component of $G_2$ that has exactly one node less than $Q$.
By moving $x_{t}$ to cover the extra node in $P'$ (and shifting the place of the fire sources used for covering $Q$ without changing their order, if it is necessary), we find a covering for $G_2$ by $t$ closed neighbourhoods with restricted radii. Hence, again in this case, $b(G_2) \leq k-1$.

Now, assume that $L_s \setminus N_{k-1}[x]$ is empty, and $G_1 = G \setminus N_{k-1}[x]$, and $G_2 = G \setminus N_{k-1}[x']$. If $L_s \setminus N_{k-1}[x']$ is empty, then $G_2$ is an isometric subforest of $G_1$, and therefore $b(G_2) \leq b(G_1) \leq k-1$.

If $L_s \setminus N_{k-1}[x']$ is non-empty, then it means that $P' = L_s \setminus N_{k-1}[x']$ contains exactly one node. Also, we know that all the non-empty components of $G_2$ are subsets of the corresponding components of $G_1$, with exactly one less node. Assume that $(x_1, x_2, \ldots, x_t)$ is an optimum burning sequence for $G_1$.
Since, $L_s \setminus N_{k-1}[x]$ is empty, then there must be non-empty component of $G_1$ like $Q$ for which $x_t \in Q$. By moving $x_{t}$ to cover the extra node in $P'$ (and shifting the place of the fire sources used for covering $Q$ without changing their order, if it is necessary), we find a covering for $G_2$ by $t$ closed neighbourhoods with restricted radii. Hence, again in this case we conclude that $b(G_2) \leq k-1$.
\end{proof}

As a consequence of the above lemma we have the following result.

\begin{figure}[H]
\centering
\begin{subfigure}{.35\textwidth}
  \centering
  \includegraphics[width=.45\linewidth]{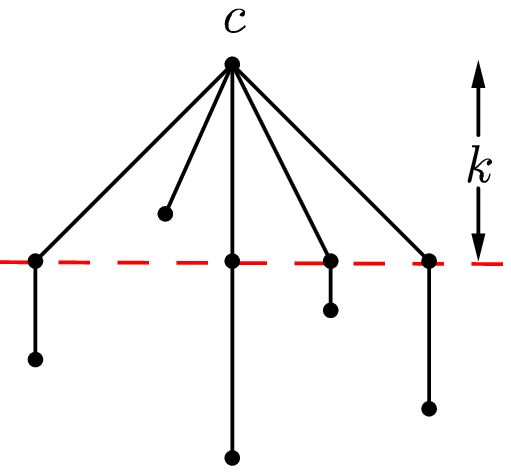}
  \caption{Part $(i)$}
\end{subfigure}%
\begin{subfigure}{.35\textwidth}
  \centering
  \includegraphics[width=.45\linewidth]{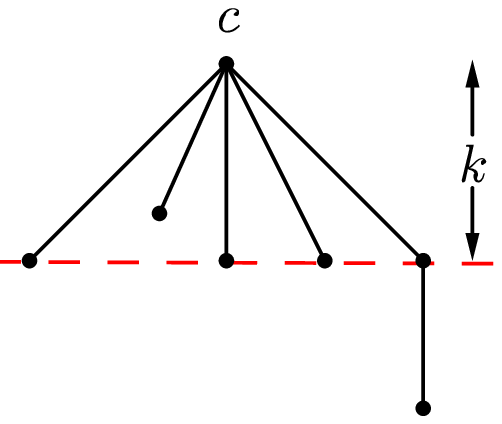}
  \caption{Part $(ii)$}
\end{subfigure}
\begin{subfigure}{.35\textwidth}
  \centering
  \includegraphics[width=.45\linewidth]{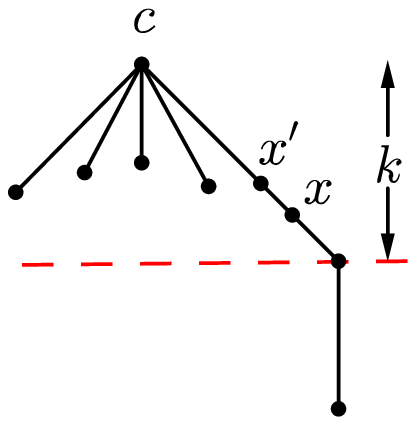}
  \caption{Part $(iii)$}
\end{subfigure}
\caption{}\label{head}
\end{figure}
\begin{lemma}\label{burningspiderhead}
Let $G$ be a spider graph with spider head $c$.
Also, suppose that for a positive integer $k$ and a node $x$ in $G$, $G\setminus N_{k-1}[x]$ is a non-empty path-forest (that is, $d(x,c) \leq k-1$)  with at least one component, and  $b(G \setminus N_{k-1}[x]) \leq k-1$. If $x\neq c$, then we have one of the following possibilities:

$(i)$ The graph $G\setminus N_{k-1}[c]$ has at least two components, and $b(G\setminus N_{k-1}[c]) \leq k-1$.

$(ii)$ There is a leaf in $G$ that is of distance $k-1$ from $c$, and $G\setminus N_{k-1}[c]$ has only one component, and $b(G\setminus N_{k-1}[c]) \leq k-1$.

$(iii)$ There is a node $x' \not\in \{ x, c\}$ on the path connecting $x$ to $c$ for which $G \setminus N_{k-1}[x']$ has only one component (that is a subset of $L_s$), and $b(G \setminus N_{k-1}[x']) \leq k-1$, and there is a leaf of $G$ that is of distance $k-1$ from $x'$.
\end{lemma}

\begin{proof}
Since the spider graph $G$ satisfies the conditions in Lemma \ref{key2}, by applying Lemma \ref{key2} for a finite number of times we derive the desired result.
In Figure \ref{head} we see a layout of the three different cases stated in the lemma.
\end{proof}

A \emph{perfect spider of radius $r$} is a spider graph $G$ with a unique centre node $c$ such that $d(v, c) = r$ for every leaf $v\in G$. We denote the set of all perfect spider trees of radius $k$ with $\mathrm{PS}_k$.

A \emph{$k$-burning maximal spider graph}, is a spider graph with spider head $c$ that its node set can be decomposed into a perfect spider graph $F = N_{k-1}[s] \in \mathrm{PS}_{k-1}$, where $s$ is a node with $d(s,c) \leq k-1$, and a graph  $H \in \mathrm{MPF}_{k-1}$. We denote the set of all $k$-burning maximal spider graphs by $k$-$\mathrm{BMS}$.
By above Lemma, we can see that there are two different types of the graphs in $k$-$\mathrm{BMS}$ like $G$: either $G$ is a graph for which the centre node $s$ of the perfect spider graph in the decomposition of $G$ is the spider head, or $G$ is a graph such that the centre node $s$ is not the spider head. If the latter holds, then by Lemma \ref{burningspiderhead} part $(iii)$, we conclude that the path-forest $G\setminus N_{k-1}[s]$ must be a single path of order $(k-1)^2$.

Note that the path-forest that appears in decomposing a $k$-$\mathrm{BMS}$ forms a decomposed spider as described in Conjecture \ref{lowk}. Now, we have the following useful theorem that also shows the truth of Conjecture \ref{lowk} for $k$-$\mathrm{BMS}$ trees.

\begin{theorem}\label{kbms}
If $G$ is a $k$-$\mathrm{BMS}$ with spider head $c$, then $b(G) = k$. 
\end{theorem}

\begin{proof}
Let $G$ be a $k$-$\mathrm{BMS}$ with spider head $c$. If the centre of the perfect spider in decomposing $G$ equals $c$, then it implies that $b(G\setminus N_{k-1}[c]) = k-1$.
In this case, $G_1 = G \setminus N_{k-1}[c]$ is in $\mathrm{MPF}_{k-1}$. By contradiction suppose that $b(G) = t \leq k-1$. Thus, by Theorem \ref{key}, there must be an optimum burning sequence for $G$ like $(x_1, x_2, \ldots, x_t)$ such that $d(x_1 , c)\leq t-1\leq k-2$.

If $x_1 =c$, then $G\setminus N_{k-1}[c]$ is an isometric subforest of $G\setminus N_{t-1}[c]$, and therefore, $b(G\setminus N_{t-1}[c]) \leq t-1\leq k-2$, which is a contradiction.

If $x_1\neq c$, then note that all the arms of $G$ are of length at least $k-1$, and since $b(G\setminus N_{k-1}[c]) = k-1$, there must be at least one arm of $G$ that is of length at least $k-1 + 2k-3 = 3k-4$. Thus, $G\setminus N_{t-1}[c]$ must have at least two non-empty components, and therefore, by Lemma \ref{key2}, we have that $b(G\setminus N_{t-1}[c]) \leq t-1\leq k-2$, which is a contradiction, as $G\setminus N_{k-1}[c]$ is an isometric subforest of $G\setminus N_{t-1}[c]$.
Hence, in both cases we find a contradiction, and therefore, $b(G) \geq k$.

If the centre of the perfect spider in decomposing $G$ is a node $s \neq c$, then as we discussed before the theorem's statement, the graph $G_1 = G \setminus N_{k-1}[s]$ is a single path of order $(k-1)^2$. Let $L_s$ be the arm of $G$ with $s\in L_s$, and assume that $v_s$ is the node in $L_s$ such that $d(s,v_s) = k-1$. Also, assume that $P$ is the path between $v_s$ and $s$, and $P'$ is the path connecting $s$ to $c$ excluding $s$.

By contradiction suppose that $b(G) = t \leq k-1$. Thus, by Theorem \ref{key}, there must be an optimum burning sequence for $G$ like $(x_1, x_2, \ldots, x_t)$ such that $d(x_1 , c)\leq t-1\leq k-2$. We consider different possibilities for $x_1$ as follows:

If $x_1$ is in $G\setminus (L_s \setminus P')$, then clearly $G\setminus N_{k-1}[s]$ is an isometric subforest of $G\setminus N_{t-1}[c]$, and therefore, we must have $b(G_1) \leq t-1\leq k-2$, which is a contradiction.

If $x_1$ is in $P$, then let $x'$ be the neighbour of $s$ on the path connecting $s$ to $c$. Note that all the leaves of $G$, except for the leaf in $L_s$, are of  distance $k-1$ from $s$. Thus, $G\setminus N_{t-1}[x_1]$ must have at least two non-empty components, and therefore, by applying Lemma \ref{key2} for a finite number of times, we have that $b(G\setminus N_{t-1}[x']) \leq t-1\leq k-2$, which is a contradiction, as $G\setminus N_{k-1}[s]$ is an isometric subforest of $G\setminus N_{t-1}[x']$.
Hence, in both cases we find a contradiction, and therefore, $b(G) \geq k$.
\end{proof}

By the arguments in proof of Theorem~\ref{kbms}, we can conclude that the spider graphs in $2$-$\mathrm{MBS}$ can be decomposed into a perfect spider $SP(s,1)$  and a single node $ P_1$, with $s\geq 3$; that is, a spider with $s-1$ arms of length one and an arm of length two.
Now, we can present an algorithm for finding the burning number of a spider tree, as follows.
Note that the burning number of every spider graph is at least two.
We denote the set of all perfect spider trees of radius $k$ with $t$ arms by $\mathrm{PS}^t_k$.
We denote the set of all $k$-burning maximal spider graphs with $t$ arms by $k$-$\mathrm{BMS}^t$.

\begin{algorithm}\label{alg2}
Suppose that $G$ is a spider tree with arms $L_1,L_2, \ldots ,L_t$, for a constant $t\geq 1$, such that the length of each arm $L_i$ is denoted by $l_i$, and $l_1\geq l_2 \geq \cdots \geq  l_t$. Let $m$ be a positive integer for which $l_1 \leq m$; that is, the length of each arm in $G$ is bounded above by $m$. Then we perform the following steps until $G \subseteq H$, for some $H\in k\text{-}\mathrm{BMS}^t$ where $k\geq 2$.

\medskip

{\bf Stage 1.} For the initial case $k=2$, we put the graph $SP(t,1)$  in $\mathrm{PS}^t_1$.
Then we add a single node to one of the arms in $SP(t,1) \in \mathrm{PS}^t_1$, and we put the resulting graph $H$ in $2$-$\mathrm{MBS}^t$. 

If $G\subseteq H$, then return $b(G) = 2$; otherwise, go to Stage 2.

\medskip

{\bf Stage 2.} For $k \geq 3$, we perform the following steps:

\medskip

{\bf Stage 2.1.} For $0\leq i \leq k-2$, we make a spider graph with $t-1$ arms of length $k-1 - i$, and then we add an additional arm of length $i+ k-1$ to it. We call the resulting spider (with $t$ arms) by $H_i$ and we put it in $\mathrm{PS}^t_{k-1}$. 

\medskip

{\bf Stage 2.2.} For $1\leq s \leq k-1$, and each $F \in \mathrm{MPF}^s_{k-1}$ (generated by Algorithm \ref{alg1} for the graph $G' = (l_1, l_2, \ldots, l_t )$), we join an end point of each component of $F$ to a distinct leaf of $H_0 \in \mathrm{PS}^t_{k-1}$, and we call the resulting graph by $F'$.
Then we add $F'$ to $k$-$\mathrm{MBS}^t$.

If $G \subseteq F'$, then stop and return $b(G) = k$.

\medskip

{\bf Stage 2.3.} For $1\leq i \leq k-2$, we join the end point of longest arm of $H_i$ to a path of order $(k-1)^2$ in $\mathrm{MPF}^1_{k-1}$, and we call the resulting graph by $H'_i$.
Then we add $H'_i$ to $k$-$\mathrm{MBS}^t$.

If $G \subseteq H'_i$, then stop and return $b(G) = k$.
\end{algorithm}

If Algorithm \ref{alg2} stops at $i=k$, then it means that  $G$ is a subgraph of a graph in $k$-$\mathrm{MBS}^t$. By Theorem \ref{kbms}, we know that the burning number of a graph in $i$-$\mathrm{MBS}^t$ equals $i$. Hence,  by Corollary \ref{subforest} from \cite{thesis}, we conclude that $b(G) = k$.
We have the following theorem about the complexity of Algorithm~\ref{alg2}.

\begin{theorem}\label{polyalg2}
Suppose that $G $ is a spider tree with arms $L_1,L_2, \ldots ,L_t$, for a constant $t\geq 1$, in which the length of each arm $L_i$ is denoted by $l_i$, and $l_1\geq l_2 \geq \cdots \geq  l_t$. Let $m$ be a positive integer for which $l_1 \leq m$; that is, the length of each arm in $G$ is bounded above by $m$. If $t $ is a constant in terms of $m$, then Algorithm \ref{alg2} finds the burning number of $G$ in polynomial time in terms of the input.
\end{theorem}

\begin{proof}
Given the graph $G$, suppose that for some $k\geq t$, Algorithm \ref{alg2} stops by recognizing $G$ as a subgraph of a graph in $\mathrm{MBS}^t_k$; that is, $b(G) =k$. In Algorithm \ref{alg2}, we first generate all the perfect spider graphs of radius $i$ with $t$ arms, for $1\leq i \leq k$. Then at Stage 2.2, we need to perform Algorithm \ref{alg1} for the graph $(l_1, l_2, \ldots, l_t)$ which satisfies all the conditions in Theorem \ref{polyalg1}. Hence, we perform at most $O(tm^t)$ steps to find all the maximal path-forests generated by Algorithm \ref{alg1} at Stage 2.2.

On the other hand, we know that $k = b(G) \leq \mathrm{radius}(G) + 1 \leq m+1 = O(m)$. Thus, $\frac{k(k+1)}{2} = O(m^2)$.
Note that the number of the perfect spider graphs that we generate in Algorithm \ref{alg2} for each $1\leq i\leq k$ equals $i$. Therefore, the total number of the graphs that we create and consider by Algorithm \ref{alg2} is asymptotically of order
$$\sum_{i=1}^k k O(tm^t) = \frac{k(k+1)}{2} O(tm^t) = O(tm^{t+2}).$$
Finally, note that each time that we add a new spider graph $F$ to $\mathrm{MBS}^t_k$, for $k\geq 2$, we compare $G$ with $F$. We can simply do this comparison by comparing the lengths of the arms between $G$ and $F$. Since, $G$ and $F$ both have $t$ arms, then the total number of the steps that we perform in Algorithm \ref{alg2} is bounded above by $O(t^2m^{t+2})$.
Since $t$ is a fixed constant in terms of $m$, then Algorithm \ref{alg2} is a polynomial time algorithm in terms of the input.
\end{proof}

\end{document}